\documentclass[11pt,a4paper,reqno]{amsart}
\usepackage{amsmath}
\usepackage{amsfonts}
\usepackage{amssymb}
\usepackage{amscd}
\usepackage{url}
\usepackage[pdftex,bookmarks=true]{hyperref}

\newcommand\F{{\mathbf{F}}}
\newcommand\OL{\mathcal OL}

\newcommand\CE{{\mathcal E}}
\newcommand\CF{{\mathcal F}}

\newcommand\ep{\epsilon}

\theoremstyle{plain}
  \newtheorem{theorem}[subsection]{Theorem}
  \newtheorem{conjecture}[subsection]{Conjecture}

  \newtheorem{fact}[subsection]{Fact}
  \newtheorem{lemma}[subsection]{Lemma}
  \newtheorem{corollary}[subsection]{Corollary}

\theoremstyle{remark}

  \newtheorem{question}[subsection]{Question}

\theoremstyle{definition}

\include{psfig}

\begin{document}

\title[A characterization of incomplete sequences in vector spaces]
{A characterization of incomplete sequences in vector spaces}

\author{Hoi H. Nguyen}
\address{Department of Mathematics, Rutgers University, Piscataway, NJ 08854, USA}
\email{hoi@math.rutgers.edu}
\thanks{}

\author{Van H. Vu}
\address{Department of Mathematics, Rutgers University, Piscataway, NJ 08854, USA}
\email{vanvu@math.rutgers.edu}

\thanks{The authors are partially supported by research grants DMS-0901216 and AFOSAR-FA-9550-09-1-0167.}
%\subjclass{11B25}

\maketitle

\begin{abstract}
  A sequence $A$ of elements
  an additive group $G$  is {\it incomplete}  if there exists a group element that {\it can not}  be expressed as a sum of  elements from $A$. The study of incomplete
  sequences is a popular topic in combinatorial number theory. However, the structure of incomplete sequences is still far from being understood, even in basic groups.

  The main goal of this paper is to give a characterization of  incomplete sequences in  the vector space $\F_p^d$, where $d$ is a fixed integer and $p$ is a large prime. As an application, we give a new proof for a recent result by Gao-Ruzsa-Thangadurai on the Olson's constant of $\F_p^2$ and partially answer their conjecture
  concerning $\F_p^3$.

  %Our approach is general and seems to be useful for other related problems.
\end{abstract}

\section{Introduction}\label{section:introduction}

\noindent Let $G$ be an additive group and $A$ be a sequence of elements of $G$.
We denote by $S_A$ the collection of subsequence sums of $A$:

$$
S_A= \left \{ \sum_{x\in B}x |B\subset A, 0<|B| < \infty \right \}.
$$

 For a positive integer $m\le |A|$, we denote by $m^*A$ the
collection of partial sums over subsequences of cardinality $m$,

$$
m^{\ast}A := \left \{ \sum_{x\in B}x |B\subset A, |B|=m \right \}.
$$

 If $0\not \in S_A$ (or $0\not \in m^\ast A)$ then we say that $A$ is {\it zero-sum-free} (or $m$-{\it zero-sum-free}).
If $S_A \neq G$ (or $m^\ast A \neq G) $ then we say that $A$ is {\it incomplete}
(or $m$-{\it incomplete}).

 The following questions are among the most popular  in
classical combinatorial number theory.

\begin{question} \label{question:1} When  is $A$ zero-sum-free? When is $A$ incomplete? When is $A$ $m-$zero-sum-free? When is $A$ $m$-incomplete?
\end{question}

  There is a number of results concerning these
questions (see for instance \cite{GG,Ge,Sarbook,Sun, Vu} for surveys) including classical results such as Olson's theorem and the Erd\H{o}s-Ginzburg-Ziv theorem. Our goal
 is to study the above problems for the basic group $\F_p^d$, as $d$ is fixed and $p$ is a large prime. (Here and later $\F_p$ denote the finite field with $p$ elements.)

Our understanding in the case $d=1$  is more or less satisfying, due to the results from \cite{NgSzV, NgV} (see also \cite{Vu} for a survey).
However, the proofs of these results do not extend to higher dimensions. The main difficulty in the extension is the existence of  non-trivial subgroups (subspaces).
In this paper, we develop a new approach that leads to a characterization for incomplete sequences in  $\F_p^d$ for $d\ge 2$.  This approach makes important use of
ideas developed by Alon and Dubiner in \cite{AD}.

In what follows, it is important to distinguish {\it sequence}  (which means multiple set, where an element may have multiplicity greater than one) and  {\it set} or {\it subset} (each element appears exactly once).

Let us start by a  simple observation.
As $S_A = \cup_{1\le m\le |A|} m^\ast A$,  if $A$ is incomplete then it is also $m$-incomplete for every $m\le |A|$. So, we  are going  to consider
  $m$-incomplete  sequences. It is clear that if $A$ belongs to a translate of a proper subspace of $\F_p^d$, then $m^\ast A$ belongs to another translate of that proper subspace, and hence $A$ is $m$-incomplete.

Our leading intuition is that some sort of  converse statement  must hold. Roughly speaking, we expect that  the main reason
for a sequence $A$ in $\F_p^d$ to be  $m$-incomplete is that its elements are contained in few translates of a
proper subspace. In the special case $d=1$, the only proper subspace is $\{0\}$. Thus if $A$ is a $m$-incomplete sequence in $\F_p$, then  it consists of few  elements of high multiplicities (see \cite{NgV} for detailed discussion). We are able to quantify this intuition in the following form.

\begin{theorem}[Characterization for incomplete sequences in $\F_p^d$]\label{theorem:characterization1} For any positive integer $d$ and
positive constants $\alpha, \beta$ there is a positive constant
$\ep$  such that the following holds for every sequence $A$ in $\F_p^d$, where $p$ is a sufficiently large prime. Either there is an $ m \le
\beta p$ such that $m^{\ast} A= \F_p^d$  or $A$ can be partitioned
into disjoint subsequences $A_0, A_1, \dots, A_l$ such that

\begin{itemize}

\item $|A_0| \le \alpha  p ;$

\item $|A_i| = \lfloor \ep p \rfloor, i\ge 1;$

\item there is a subspace $H$ such that each $A_i$ ($1\le i\le
l$) is contained in a translate of $H$ and  $m^{\ast} {A_0 }$
contains a translate of $H$ for some $1\le m \le |A_0|$.

\end{itemize}

\end{theorem}

The set $A_0$ is often viewed as the set of exceptional elements. Note that if $H=\{0\}$, then each $A_i$ consists of only one element of multiplicity $\lfloor \ep p \rfloor$.

%We can tighten the above result a bit more, namely that the translates of $H$ that contain $A_i$ must be close to the origin. However, we do not elaborate on that point here.

%Also, a slightly more general statement is true for sequences whose subset sums do not complete the whole $\F_p^d$.

%\begin{theorem}[characterization for incomplete sequences in $\F_p^d$, II]\label{theorem:characterization2} For any positive integers
%$d,n$ and positive constants $\alpha, \beta$ there is a positive
%constant $\ep$ such that the following holds. If there is no $1 \le
%m_i \le |A_i|, \sum_{i=1}^n m_i \le \beta p$ such that $\sum_{i=1}^k
%{m_i}^{\ast} A_i= \F_p^d$ then $A_i$ can be partitioned into
%disjoint subsequences $A_i^j$ such that

%\begin{itemize}

%\item $\sum_{i=1}^n |A_i^0| \le \alpha p;$

%\item $|A_i^j| = \lfloor \ep p \rfloor$ for $1\le i\le n, 1\le j;$

%\item there is a subspace $H$ such that each $A_i^j$ is contained in a translate of $H$. Furthermore there exist $1\le m_i \le |A_i^0|$ such that $\sum_{i=1}^n m_i^{\ast} {A_i^0 }$ contains
%a translate of $H$.

%\end{itemize}

%\end{theorem}

Theorem \ref{theorem:characterization1} seem to be  applicable for  various  additive problems in vector spaces.  In the rest of this section, we discuss a few examples.

\vskip2mm

{\bf Olson's constant.} Given an additive group $G$, let $\OL(G)$, the {\it Olson constant} of $G$, be the smallest integer such that no subset of $G$ of cardinality $\OL(G)$ is zero-sum-free.
 $\OL(G)$ is a parameter of principal interest and its determination has  a long history.  Erd\H{o}s, Graham and Heilbronn (\cite{EG}) conjectured that $\OL(G)\le \sqrt{2|G|}$.  Szemer\'edi (\cite{Sz}) showed that there exists an absolute constant $c$ such that $\OL(G)\le c\sqrt{|G|}$. Olson (\cite{O1},\cite{O2}) then improved it to $\OL(G)\le 2\sqrt{|G|}$. The current record is due to Hamidoune and
Z\'emor (\cite{HZ}) who showed that $\OL(G)\le \sqrt{2|G|}+O(|G|^{1/3}\log{|G|})$.

Regarding the group $\F_p$, Hamidoune and Z\'emor proved that $\OL(\F_p)\le \lceil \sqrt{2p}+5\log{p}\rceil$. They conjectured that the additional $\log$ term is not necessary, and this has been recently settled by the current authors with Szemer\'edi (also due to Deshouillers and Prakash, see \cite{NgSzV}).
These results give the exact value of $\OL (\F_p)$.  Recently, Gao, Ruzsa and Thangadurai \cite{GRT} showed that $\OL(\F_p^2) = p+\OL(\F_p)-1$.

%The proof presented in \cite{GRT} gives an explicit lower bound for $p$, but seems to fall short for the task of classifying zero-sum-free sets of cardinality close to $\OL(\F_p^2)$. One of our application is to handle this %problem, which also implies  the result of G-R-T as a special case.

Our first application is the following strengthening of Gao-Ruzsa-Thangadurai result.

\begin{theorem}[Description of optimal zero-sum-free sets in $\F_p^2$] \label{theorem:Olson2}
Suppose that $A$ is a zero-sum-free set of cardinality
$p+\OL(\F_p)-2$ in $\F_p^2$, where $p$ is a sufficiently large prime.
Then there is a linear full rank map $\Phi$ such that one of the following holds.

\begin{itemize}

\item $\Phi(A)$ contains $\OL(\F_p)-1$ points on the line $x=0$ and $p-1$ points on
the line $x=1$.

\vskip .1in

\item $\Phi(A)$ contains $\OL(\F_p)-1$ points on the line $x=0$, $p-2$ points on
the line $x=1$, and one point on the line $x=2$.

\end{itemize}

\end{theorem}

Theorem \ref{theorem:Olson2} not only reproves the bound $p + \OL(\F_p) -2$, but also characterizes all extremal sets. The proof  is  short and simple; furthermore, it also classifies zero-sum-free sets of size $\ge \ep p$, for any constant $\ep >0$, but we do not elaborate on this point.

It is not hard to see that there exists in $\F_p^d$ a zero-sum-free set of size $p+\OL(\F_p^{d-1})-2$. (One can see this by taking the union of a zero-sum-free set of size $\OL(\F_p^{d-1})-1$ on the hyperplane $x_d=0$ and $p-1$ arbitrary points of the plane $x_d=1$.) We thus have $\OL(\F_p^{d}) \ge p+\OL(\F_p^{d-1})-1$.
Gao, Ruzsa and Thangadurai \cite{GRT}  made the following conjecture.

\begin{conjecture}[Precise value of Olson's constant for $\F_p^d$] For any fixed $d$ and sufficiently large $p$,
$\OL(\F_p^{d}) = p+\OL(\F_p^{d-1})-1$.
\end{conjecture}

The assumption that  $p$ is sufficiently large is necessary,  see \cite{GG}.

Since $\OL (\F_p ) = O(\sqrt p) =o(p)$, this conjecture would imply the following asymptotic version.

\begin{conjecture} \label{conj2}  For any fixed $d \ge 2$ and $\gamma >0$, the following holds for all sufficiently large $p$:
$\OL(\F_p^{d}) \le  (d-1 + \gamma ) p.$ (Notice that the lower bound $\OL (\F_p^d ) \ge (d-1-\gamma) p $ is obvious.)
\end{conjecture}

 As far as we know (prior to this paper)  there has been no progress on either  conjecture. As  another application of Theorem \ref{theorem:characterization1}
 we settle  the $d=3$ case of Conjecture \ref{conj2}.

\begin{theorem}[Asymptotic value of Olson's constant for $\F_p^3$]\label{theorem:Olson3} Let $\gamma>0$ be an arbitrary positive constant, then the following holds for sufficiently large prime $p$,

$$\OL(\F_p^3)\le (2+\gamma)p.$$

\end{theorem}

It is possible that one can establish Conjecture \ref{conj2} for arbitrary $d$ using this approach, but
the technical details still elude us at this point.

\vskip2mm

%{\bf To Hoi: Meshulam's problem.} {\bf to be written}.

In the rest of this section we introduce our notation. The
remaining sections are organized as follows. In Section \ref{section:lemmas} we provide our main lemmas. The proof of the characterization result is presented in Section \ref{section:characterization1}. The last two sections are devoted to the proofs of Theorem \ref{theorem:Olson2} and Theorem \ref{theorem:Olson3}, respectively.

\vskip .4in

\centerline{\bf Notation}

\vskip .2in

{\it Norm in $\F_p$}. For $x \in \F_p$, $\|x\|$ (the
norm of $x$) is the circular distance from $x$ to $0$. (For example, the norm
of $p-1$ is 1.)

{\it Dilation of a sequence}. For $b \in \F_p$ and a sequence $A\subset \F_p$,
$b \cdot A$ is the collection of all $ba$, where $a$ varies in $A$.

{\it Projections}. Let $H$ and $H'$ be (not necessarily orthogonal) complementary subspaces. For a vector $a\in \F_p^d$ we let
$\pi_H(a),\pi_{H'}(a)$ be the unique vectors
$a_H,a_{H'}$ satisfying $a=a_H+a_{H'}$
and $a_H\in H, a_{H'}\in H'$ respectively. (Whenever we use this notation, both $H$ and $H'$ will be specified.)
For a given sequence $A$, set  $\pi_H(A):=\{\pi_H(a)|a\in A\}.$

{\it Affine basis}.  A collection of $d+1$ vectors in  general position  forms an {\it affine basis} of $\F_p^d$.

\section{Technical Lemmas}\label{section:lemmas}

We are going to  use  the following results from \cite{AD}.

\begin{lemma}[Sumset of affine bases, \cite{AD}] \label{lemma:AD1} Assume that $s \le p$. Let $A_1, \dots, A_s$ be $s$
affine bases of $\F_p^d$. Then

$$|A_1 + \dots + A_s | \ge (\frac{s}{8d})^d. $$

\end{lemma}

\begin{lemma}[Linear independence implies growth in sumset, \cite{AD}] \label{lemma:AD2} Let $W$ be a number at least one
and $A$ be a sequence in $\F_p^d$ such that no hyperplane contains
more than $|A|/(4W)$ elements of $A$. Then for every subset $Y$ of
$\F_p^d$ of size at most $p^d/2$, there is an element $a$ of $A$
such that

$$|(a+Y) \backslash Y| \ge \frac{W}{16p} |Y|. $$
\end{lemma}

One can immediately derive the following corollary

\begin{corollary} \label{cor:AD2}

Let $W$ be a number at least one and $A$ be a sequence in $\F_p^d$
such that no hyperplane contains more than $|A|/(4W)$ elements of $A$.
Then for every subset $Y$ of $\F_p^d$ of size at most $p^d/2$ and
any element $a' \in  A$ there is another element $a \in A$ such that

$$|(a+Y) \backslash (a'+Y)| \ge \frac{W}{16p} |Y|. $$
\end{corollary}

\begin{proof} (of Corollary \ref{cor:AD2}) Apply Lemma \ref{lemma:AD2} for
 the sequence $A' := \{x-a'| x\in A \}$ and $Y$. If there is a
 hyperplane $H$ containing $|A'|/(4W) = |A|/(4W)$ elements of
$A'$, then the hyperplane $H':=a'+H$ contains $|A|/(4W)$ elements of
$A$. If there is no such $H$, then there is an element $a-a'$ in
$A'$ such that

$$|((a-a') + Y) \backslash Y| \ge \frac{W}{16p} |Y|. $$

 Notice that $y \in ((a-a') + Y) \backslash Y$ if and only
if $a'+ y \in (a+Y)\backslash (a'+Y)$. The claim follows.
\end{proof}

Our proof for the characterization theorem is based on induction on $d$. We will invoke the following result from our earlier paper \cite{NgV}  as a black box.

\begin{theorem}[Characterization for incomplete sequences in $\F_p$, \cite{NgV}]\label{theorem:classification:d=1}
Assume that $A$ is an incomplete sequence in $\F_p$, with
sufficiently large $p$. Then there is a residue $b\neq 0$ such that
we can partition $b\cdot A$ into two disjoint subsequences $b\cdot A
= A^\flat \cup A^\sharp$ where
\begin{itemize}
\item $|A^\flat| \le p^{12/13},$
\item $\sum_{a\in A^\sharp} \|a\| < p.$
\end{itemize}
\end{theorem}

We close this section with a trivial, but useful,  fact.

\begin{fact}\label{fact:dimensionincrement}
Let $H_1,H_1'$ be subspaces such that $H_1\oplus H_1' =\F_p^d$. Let $A_1,A_2$ be such sequences in $\F_p^d$ that $m_1^\ast A_1$
contains a translate of a subspace $H_1$ of $\F_p^d$ and $m_2^\ast
(\pi_{H_1'}(A_2))$ contains a translate of a subspace $H_2$
of $H_1'$. Then $m_1^\ast A_1 + m_2^\ast A_2$ contains a
translate of the subspace $H_1+H_2$ in $\F_p^d$.
\end{fact}

\begin{proof} (of Fact \ref{fact:dimensionincrement})
Since $m_1^\ast A_1$ contains a translate of $H_1$, there exists a vector $v_1\in \F_p^d$ such that $v_1+H_1 \subset m_1^\ast A_1$.

Since $m_2^\ast (\pi_{H_1'}(A_2))$ contains a translate of $H_2$, there exists a vector $v_2\in \F_p^d$ such that for any $h_2\in
H_2$ there corresponds a vector $h_1\in H_1$ satisfying $v_2+h_1+h_2\in m_2^\ast A_2$. Hence  $v_1+H_1 +v_2+h_1+h_2 = v_1+v_2+H_1+h_2 \subset m_1^\ast A_1 + m_2^\ast A_2$.

Since the above holds for all $h_2\in H_2$, we have $v_1+v_2 + H_1+H_2 \subset m_1^\ast A_1 + m_2^\ast A_2.$

\end{proof}

\section{Proof of Theorem \ref{theorem:characterization1}}\label{section:characterization1}

\subsection{Existence of rich hyperplanes}

\begin{lemma}[Rich hyperplane lemma] \label{lemma:AD3}
For any positive integer $d$ and positive constants  $\beta, \delta$
there is a positive constant $\ep$ such that the following holds.
Let $A$ be a sequence in $\F_p^d$ with at least $\delta p$ elements so that there is no $1\le m \le \beta p$ such that $m^{\ast} A=
\F_p^d$. Then, there is a hyperplane $H$ such that

$$ |A \cap H| \ge \ep p. $$

\end{lemma}

\begin{proof}(of Lemma \ref{lemma:AD3}) Let $c_1\le \min \{\beta/4(d+1), \delta/4(d+1) \}$ be a small positive constant. We consider
two cases.

{\bf Case 1.} One cannot find  $2c_1p$ disjoint affine bases in $A$.

 In this case, by the definition of affine basis, there is
a hyperplane $H$ containing
$$|A|- 2c_1 p(d+1) \ge \frac{\delta}{2}p$$ elements of $A$ and we
are done.

\vskip2mm

 {\bf Case 2.} One can find $ 2c_1p$ disjoint affine
bases in $A$.

 Set $s:=c_1p$ and let $E_1, \dots, E_s, F_1, \dots, F_s$
be the bases. Define

$$\CE_0 := E_1 + \dots + E_s   \mbox{  and  }  \CF_0:= F_1 + \dots + F_s. $$

 By Lemma \ref{lemma:AD1},

    $$\min \{|\CE_0|, |\CF_0| \}\ge (c_1p/8d)^d.$$

 Partition $A \backslash
(\cup_{i=1}^s E_i  \cup \cup_{i=1}^s F_i)$ into two sequences
$E^{\dag}, F^{\dag}$ of equal sizes (we can throw one element from
$A$ to ensure parity). By choosing $c_1$ sufficiently small, we can
assume that $|E^{\dag}| = |F^{\dag}| \ge |A|/4$.

 Let $W= W(d, \delta, \beta)$ be a large constant and $\ep$
be a small constant to be determined. Assume, for a contradiction
that there is no hyperplane containing $\ep |A|$ elements of $A$.

 Let $a_0'$ be an arbitrary element of $E^{\dag}$. By the
way we set $\ep$ so that

\begin{equation}\label{eqn:ep}
\ep |A| \le (|E^{\dag}|-\delta p/8)/ (8W).
\end{equation}

 Thus  there is no hyperplane containing $|E^{\dag}|/(8W)$
elements of $E^{\dag}$. By Corollary \ref{cor:AD2}, we find $a_0 \in
E^{\dag} \backslash \{a_0'\}$ such that

$$|(a'_0 + \CE_0) \cup (a_0+ \CE_0)| \ge (1+ \frac{W}{16p}) |\CE_0|. $$

 Define $\CE_1= (a_0' + \CE_0) \cup (a_0+\CE_0)$ and
$E^{\dag}_1 := E^{\dag} \backslash \{a_0', a_0\} $. Repeating  the
argument, we find elements $a_1', a_1 \in E^{\dag}_1$ such that

$$|(a_1' + \CE_1 ) \cup (a_1+ \CE_1) | \ge (1+ \frac{W}{16p}) |\CE_1|. $$

 In general, set $E^{\dag}_i:= E^{\dag}_{i-1} \backslash
\{a'_{i-1}, a_{i-1} \}$ and $\CE_i := (a'_{i-1} + \CE_{i-1}) \cup
(a_{i-1}+ \CE_{i-1})$. By induction, we have

$$|\CE_i|  \ge (1+ \frac{W}{16p}) ^i |\CE_0|, $$

 unless $|\CE_{i-1} | > p^d/2$. Thus by choosing $W$
sufficiently  large (in terms of $d,\delta$ and $\beta$), there is some $0\le k \le \min \{\beta p/4, \delta
p/16 \}$ such that

$$|\CE_k| > \frac{1}{2} p^d. $$

 Notice that in every step, the condition $\ep |A| \le
|E^{\dag}_i|/(4W)$ is satisfied because of \eqref{eqn:ep} and $|E^\dag_i|\ge |E^\dag| -2k \ge |E^\dag |-\delta p/8$.

Repeating the argument with
$F^{\dag}$ and $\CF_0$, we have for some $0\le l \le \min \{\beta p/4,
\delta p/16 \}$ that

$$|\CF_l| >  \frac{1}{2} p^d. $$

 Observe that if $X$ and $Y$ are two subsets of a finite
Abelian group $G$ and $|X|, |Y| > |G|/2$, then $X+Y= G$. Thus,

$$\CE_k + \CF_l = \F_p^d. $$

 The left hand side is a subset of $m^{\ast}A$ for some
small $m $. Indeed,  the elements in $\CE_k$ (or $\CF_l$) are sums of
exactly $c_1p+k$ (or $c_1p+l$) elements of $A$. Furthermore, by the
procedure, the sequence of elements of $A$ involved in $\CE_k$ is
disjoint from the sequence of elements of $A$ involved in $\CF_l$.
Finally,

$$m \le  2c_1p + (k+l) \le p (2c_1 + \beta/4 + \beta/4) \le \beta
p. $$

 This concludes the proof of Lemma
\ref{lemma:AD3}.

\end{proof}

\subsection{Completing the proof of Theorem \ref{theorem:characterization1}}

  Let $\delta_0$ be a small positive constant. There is a positive constant $\ep_0$
 such that (using Lemma \ref{lemma:AD3} iteratively) we can partition
 $A= B_0 \cup B_1 \cup \dots \cup B_h$, where $|B_0| \le
 \delta_0p$ and $|B_i| = \lfloor \ep_0 p \rfloor $ and each  $B_i$ ($1\le i \le h$)
 is contained in a hyperplane $D_i$.
Consider two cases:

 {\bf Case 1.} There are some $1 \le i \le h$ and $1 \le m \le \beta p/2$ such that
 $m^{\ast} B_i $ contains  $v+ D_i$, a translate of $D_i$.

  We can assume that $i=1$ and that $D_1$ is
 parallel to the subspace $H$ spanned by
 the basic vectors $e_1, \dots, e_{d-1}$.
 Let $A'$ be the projection of $A \backslash
 B_1$ into $H'$ spanned by $e_d$.

 Consider $A'$. First,  $A'$ is a sequence in $\F_p$ with $|A| -\lfloor \ep_0 p \rfloor$
 elements. Second, if there is no $1\le m \le \beta p$ such that $m^{\ast}A = \F_p^d$, then
  there is no $1\le m \le \beta p/2$ such that $m^{\ast} A' =
 \F_p$ by Fact \ref{fact:dimensionincrement}. The classification theorem for $\F_p$, Theorem \ref{theorem:classification:d=1}, implies that there
 is a subsequence $A^{''}$ of $A'$ with at most $\ep_0 p$ elements such that
 $A' \backslash A^{''}$ contains $M= O_{\ep_0}(1)$ different
 elements $a_1, \dots, a_M$. Indeed, one can set $A''$ to consist of $b^{-1} \cdot A^\flat$ and those elements $b^{-1} \cdot a$, where  $a \in A^{\sharp}$ and  $\|a\|\ge 2/\ep_0$.

  Consider the system of parallel hyperplanes $ a_1+H, \dots
 a_M+H$. Set $\ep:= \alpha/ 2M$. Partition $A \cap (a_i + H)$ into
 disjoint subsequences of size exactly $\lfloor \ep p \rfloor$ and a remainder sequence of
 size less than $\ep p$. Let $A_0$ be the union of the remainders
 and $B_0,B_1$ and $A''$. We have

 $$|A_0| \le M \ep p + |B_0| +|B_1|+|A''| \le \frac{\alpha}{2} p +
 \delta_0 p + \ep_0 p + \ep_0 p \le \alpha p. $$

  Furthermore, there is some $1 \le m \le \beta p/2$ such that
$m^{\ast} B_1 \subset m^{\ast} A_0$ contains a hyperplane parallel
to $H$. Finally, $A\backslash A_0$ are partitioned into sequences of size
exactly $\lfloor \ep p \rfloor$, each of which is contained in a translate of $H$.
This concludes the proof for the first case.

 \vskip2mm

  {\bf Case 2. } There is no $1\le i \le h$ and  $1 \le m \le \beta p/2$ such that
 $m^{\ast} B_i $ is a translate of $D_i$.

  In this case, we can apply the induction hypothesis to a translate of $B_i$ (which is
 contained in the subspace parallel with $D_i$) to obtain
 a decomposition  $B_i= B_{i0} \cup B_{i1} \cup \dots \cup
 B_{il_i}$, where $B_{ij}, 1\le j \le l_i$ are contained in translates of a subspace
 $H_i$ and there is an integer $m_i$ such that  $m_{i}^{\ast} B_{i0}$ contains a translate of
 $H_{i}$ (we choose the parameters to be small enough such that $|B_{i0}| \le \min(\alpha/(2h),\beta/(2h))$).
 Without loss of generality, one may assume that all $B_{ij}$ have the same size $\lfloor \ep p \rfloor$ with a small positive constant $\ep$.

  Now take $A_0 := B_0\cup (\cup_{i} B_{i0})$ and let $A_k$ ($1 \le k \le
 \sum_{i=1}^h l_i $) be the $B_{ij}$. By choosing $\delta_0$ small, we get $|A_0| \le \alpha p$.

  Note that $m_1^{\ast} B_{10} + \dots + m_j^{\ast} B_{h0}$, and thus $(m_1+\dots+m_h)^\ast A_0$, contains a translate of $H=H_1+\dots+H_h$.
 Each $A_k$ ($1 \le k$) is contained in a translate of $H$ since it is contained in a translate of $H_k$.

%\section{Proof of Theorem \ref{theorem:characterization2}}\label{section:characterization2}

 % Set $\alpha'=\alpha /n, \beta'=\beta /n.$ Consider $A_i$. Since $m^\ast A_i \neq \F_p^d$ for $1\le m \le \beta' p$, Theorem \ref{theorem:characterization1}
 %applied to the sequences $A_i$ and the constants $\alpha', \beta'$ gives a partition $A_i=A_i^0\cup A_i^1\cup \dots \cup A_i^{h_i}$ where $|A_i^0| \le \alpha'  p , |A_i^j| = \lfloor \ep p
 %\rfloor$, $\ep$ is small, and there exists a subspace $H_i$ such that each $A_i^j$ ($1 \le j \le
 %h_i$) is contained in a translate of $H_i$. Furthermore, there exists $1\le m_i \le |A_i^0|$ such that $m_i^{\ast} {A_i^0}$
%contains a translate of $H_i$ for each $1\le i \le n$.

 % It is clear that ${m_1}^\ast A_1^0+\dots + {m_n}^\ast A_n^0$ contains a
 %translate of $H=H_1+\dots +H_n$ and $A_i^j$ is contained in a translate of
 %$H$ for any $1\le i,j$.

\section{Proof of Theorem \ref{theorem:Olson2}}\label{section:Olson2}

The idea  is to "project" the problem into $\F_p$ by using the characterization theorem. Once inside  this group,
we will be able to invoke  Theorem \ref{theorem:classification:d=1}. In fact, we will also need the following result.

\begin{theorem}[Erd\H{o}s-Heilbronn type inequality, \cite{SH}]\label{theorem:ErdosHeilbronn} Let $A$ be a non-empty subset of $\F_p$. Then $m^\ast A \ge \min\{p,m|A|-m^2+1\}.$ 
In particular, provided that $p$ is large enough, we have
 
\begin{itemize}
\item if $|A| \ge 2\sqrt{p}+1$, then $\lfloor \sqrt{p} \rfloor^\ast A=\F_p$;

\vskip .1in 

\item if $|A|\ge .51 p$, then $m^\ast A =\F_p$ for all $2\le m\le |A|-2$.
\end{itemize}
\end{theorem}

For more general results on $m-$incomplete sequences in $\F_p$, we refer the reader to \cite[Theorem 2.8]{NgV}.

\begin{fact}\label{fact:addinglines}
Let $B_1,B_2$ be subsets of the lines $x=b$ and $x=p-b$ of
$\F_p^2=\{(x,y|x,y\in \F_p)\}$ respectively. Assume that
$|B_1|+|B_2|>p$. Then $B_1+B_2$ contains the whole $y$-axis. In
particular, the set $B_1\cup B_2$ is not zero-sum-free in $\F_p^2$.
\end{fact}

Now we are ready to prove Theorem \ref{theorem:Olson2}.
Let $\alpha$ and $\beta$ be sufficiently small constants. Assume
that $A$ is zero-sum-free. By Theorem \ref{theorem:characterization1}, there exists
$A_0\subset A$ of cardinality $|A_0|\le \alpha p$ such that $m^\ast
A_0$ contains a line of $\F_p^2$. Without loss of generality, we
assume that this line is parallel to the $y$-axis
${\mathcal{L}}:x=0$.

Let $\mathcal{L}'$ be the collection of points on the $x$-axis, then $\mathcal{L}\oplus \mathcal{L}'=\F_p^2$. Let $B:=\pi_{\mathcal{L}'}(A\backslash A_0)$ be
the projection of $A\backslash A_0$ into $\mathcal{L}'$, thus $B$ is a sequence in $\F_p$. Since $A$ is zero-sum-free, we have

\begin{itemize}
\item  $|A\cap {\mathcal{L}}| <\OL(\F_p) <\sqrt{2p}+1$ (see discussion prior to Theorem \ref{theorem:Olson2}),

\vskip .2in

\item $B$ is an incomplete sequence
in $\F_p$ (by Fact \ref{fact:dimensionincrement} and that $m^\ast
A_0$ contains a translate of ${\mathcal{L}}$).
\end{itemize}

Together with Theorem \ref{theorem:classification:d=1},
the last observation implies that, after an appropriate
dilation of $A$ in the direction of the $x$-axis $\mathcal{L'}$, $B$ contains at least $(1-2\alpha) p$ elements
of norm 1 (those $b\in B$ with $\|b\|=1$). (Note
that such a dilation does not affect the property of $A_0$, i.e. $m^\ast A_0$ still contains a vertical line.)

Let $A_{-1}:=\pi_{\mathcal{L}'}^{-1}(-1)\cap A$ and $A_1:=\pi_{\mathcal{L}}^{-1}(1)\cap A$ be the collections of elements of $A$ whose $x$-coordinate are $-1$ and $1$ respectively. As these elements correspond to the one of norm 1 in $B$, we have $|A_{-1}|+|A_{1}| \ge (1-2\alpha)p$.

Without loss of generality, assume that $|A_1|\ge
2\sqrt{p}+1$. It follows from the first part of Theorem \ref{theorem:ErdosHeilbronn} that
${\lfloor \sqrt{p} \rfloor}^\ast A_1$ contains the whole line $x=\lfloor
\sqrt{p} \rfloor$, and hence $|A_{-1}|<\sqrt{p}$ by Fact \ref{fact:addinglines}. In other words, $|A_1| \ge (1-2\alpha)p -\sqrt{p} \ge (1-3\alpha)p$.

To make the presentation less technical, we abuse the notation to write $A:=A_0\cup A_1 \cup
A_{-1} \cup A'$, where the new set $A_0$ is the intersection of $A$ with the $y$-axis, $A_0:=A\cap {\mathcal{L}}$, and $A':=A\backslash (A_0\cup A_1 \cup
A_{-1}).$

Now comes a crucial observation. Since $|A_1|\ge (1-3\alpha)p$, the second part of Theorem \ref{theorem:ErdosHeilbronn} implies that $l^{*}A_1$ covers the whole vertical line $x=l$, for every $2\le l \le |A_1|-2$. Thus the set $\bigcup_{2\le l \le
|A_1|-2} l^{*}A_1$ covers the whole strip $\{(x,y)\in \F_p^2: 2\le x \le |A_1|-2\}$.

The last conclusion immediately implies that $|A_{-1}|\le 1$, otherwise $2^\ast A_{-1}+2^\ast A_1$ would contain
${\mathcal{L}}$, and hence the origin, a contradiction.

We next focus on $A'$. Let $X:=\{x_1,\dots,x_{|A'|}\}=\pi_{\mathcal{L}'}(A')$ be the projection of $A'$ into the $x$-axis $\mathcal{L}'$. It follows that there does not exist any subset of $X$ whose sum belongs to the ''opposite" of $\{2,\dots,|A_1|-2\}$, i.e. we must have $S_X\subset \{-1,0, 1,\dots, p+1-|A_1|\}$ (in $\F_p$).

This partly implies that $x_i\in \{2,\dots,p+1-|A_1|\}$, but more importantly, viewing $x_i$ as real numbers, we must have

\begin{equation}\label{equation:Olson2:1}
x_1+\dots+x_{|A'|}\le p-|A_1|+1.
\end{equation}

Indeed, if there exists an $i$ such that $x_1+\dots+x_i\le p-|A_1|+1 <x_1+\dots+x_{i+1}$ then $p-|A_1|+2 \le x_1+\dots+x_i+x_{i+1} = (x_1+\dots+x_{i}) + x_{i+1} \le p-|A_1|+1+(p-|A_1|+1) \le 6\alpha p<p/2$. Thus, after taking modulo $p$, $x_1+\dots+x_i+x_{i+1}$ remains in $\{p-|A_1|+2,\dots, 6\alpha p\}$, which belongs to the ''opposite" of $\{2,\dots, |A_1|-2\}$ in $\F_p$, a contradiction.

Since $x_i\ge 2$, \eqref{equation:Olson2:1} implies that $|A'|$ is small, $2|A'|\le p-|A_1|+1$. From the definition of $A'$, using $|A_{-1}|\le 1$, we have $2|A|-p-3 \le 2|A_0|+|A_1|$. Inserting the bound of $|A|, |A|=p+\OL(\F_p)-2$, we obtain that $2|A_0|+|A_1| \ge p+2 \OL(\F_p)-7$.

Note that $|A_0|\le \OL(\F_p)-1$, the above result trivially implies that $|A_0|+|A_1|>p$, i.e.~$A_0+A_1$ covers the whole line
$x=1$. Hence $A_{-1}$ must be empty.

Using this new information (instead of
$|A_{-1}|\le 1$) and the upper bounds $|A_0|\le \OL(\F_p)-1, |A_1|\le p-1$, we easily obtain the lower bounds $|A_1|\ge p-3$ and $|A_0|\ge \OL(\F_p)-2$, and hence \eqref{equation:Olson2:1} implies that $x_{1}+\dots +x_{|A'|} \le 4$. Since $x_i\ge 2$, we must have $|A'|\le 2$. If $|A'|=0$, we are done. It remains to consider the following two cases.

{\bf Case 1}. $|A'|=2$. We then have $x_1=x_2 =2$. Thus, if $|A_1|=p-2$ and $|A_0|= \OL(\F_p)-2$, then $(p-4)^\ast A_1+x_1+x_2$ covers the whole $y$-axis $x=0$, contradiction. Furthermore, if $|A_1|=p-3$ and $|A_0|=\OL(\F_p)-1$, then $(p-4)^\ast A_1 + x_1+x_2$ covers $p-3$ elements of the $y$-axis, and hence $(p-4)^\ast A_1 + x_1+x_2 + A_0$ covers the whole axis, another contradiction.

\vskip .1in

{\bf Case 2}. $|A'|=1$. We then easily eliminate the case $|A_1|=p-1$ and $|A_0|= \OL(\F_p)-2$ because in this case, by \eqref{equation:Olson2:1}, we must have $x_1=2$, and hence $(p-2)^\ast A_1 +x_1$ covers the whole $y$-axis.

Assuming $|A_1|=p-2$ and $|A_0|= \OL(\F_p)-1$, it is implied that $x_1 \le 3 $ by \eqref{equation:Olson2:1}. But if $x_1=3$ then $(p-3)^\ast A_1+ x_1$ covers $p-2$ elements of the $y$-axis, and hence $(p-3)^\ast A_1 + x_1+ A_0$ covers the whole axis, a contradiction. As a consequence, $x_1=2$. The only reason that $A=A_0 \cup A_1 \cup A'$ is zero-sum-free in this case is that the multiset $A_0 \cup \{\sum_{a \in A_1 \cup A'} a\}$ is zero-sum-free in $\F_p$, completing the proof.

\section{Proof of Theorem \ref{theorem:Olson3}}\label{section:Olson3}

To establish Theorem \ref{theorem:Olson3}, we again rely on our characterization theorem to project back to $\F_p^2$. After this step, we are not working with sets anymore, but rather with sequences. For this reason, we need  the following statement about the "sequence"  counterpart of Olson's constant (which is usually referred to as the Davenport's constant).

\begin{theorem}[Davenport's constant for $\F_p^d$, \cite{O1}]\label{theorem:Davenportconstant}
Any collection of $d(p-1)+1$ elements of $\F_p^d$ is not zero-sum-free.
\end{theorem}

Now we present the proof of Theorem \ref{theorem:Olson3}. Assume that there exists a set $A$ of size $(2+\gamma)p$ which is zero-sum-free. By Theorem \ref{theorem:characterization1} (after a bijective  linear mapping) we can partition $A$ into disjoint sequences, $A=A_0\cup A_1\dots\cup
 A_n$ where

 \begin{itemize}
 \item $|A_0|\le \gamma p/4$;

 \item $|A_i|= \lfloor \ep p \rfloor$;

 \item $m^\ast A_0$ contains a translate of a subspace $H$, for some $m\le |A_0|$;

 \item there exist $n \le n(\gamma)$ vectors $a_1,\dots,a_n\in \F_p^3$ such that $A_i \subset a_i + H$.

 \end{itemize}

Let $d'$ be the dimension of $H$. We observe that $d'$ can not be either 0 or 3, since the first case would imply that $A$ contains elements of multiplicity $\lfloor \ep p \rfloor >1$ (as $\ep$ is independent of $p$), while the second  would imply that $A$ is complete. We consider two remaining  cases.

{\bf Case 1.} $d'=2$. Without loss of generality, we assume that $H=\{z=0\}$. Consider the projection $B$ of $A\backslash A_0$ onto the $z$-axis, which can be viewed as a sequence in $\F_p$. Since $A$ is zero-sum-free, $|A\cap H|<\OL(\F_p^2) \le (1+\gamma/4) p$. Thus there are at least $|A|-|A_0|- (1+\gamma/4)p \ge (1+\gamma/2)p$ elements of $B$ having non-zero norm.

By item 2 of Theorem \ref{theorem:classification:d=1}, the latter implies that $B$ is complete in $\F_p$ (in fact, it is easy to see that
any sequence of $p$ non-zero elements of $\F_p$ is complete). Hence $S_{A\backslash A_0}+m^\ast A_0=\F_p^3$, in particular the origin, a contradiction.

{\bf Case 2.} $d'=1$. Without loss of generality, we assume that $H$ is the $z$-axis $\{x=0,y=0\}$. By the property of $A_i$, we may write $A_i\subset v_i+H$, where $v_i$ is the projection of $a_i$ onto the $xy$-plane $\{z=0\}$.

Consider the sequence $\{v_1,\dots,v_n\}$, where each $v_i$ has multiplicity $|A_i|-\lfloor 2/\ep \rfloor$. Since $\sum_{i=1}^n (|A_i|-2/\ep) \ge (2+\gamma/2)p$, Theorem \ref{theorem:Davenportconstant} implies that there exist $0\le m_i\le  |A_i|-2/\ep$ such that $\sum_{i=1}^n m_iv_i = 0$ (in the $xy$-plane), where not all $m_i$ are zero. By multiplying all $m_i$ by $\lfloor 2/\ep \rfloor$ if needed, we may assume that at least one of the $m_i$ is greater than $\lfloor 2/\ep \rfloor $. Let this be $m_1$.

We now consider the sumset $\sum_{i=1}^n m_i^\ast A_i$. By the definition of $m_i$ and $v_i$, the projection of this set into the $xy$-plane is the origin, in other words, $\sum_{i=1}^n m_i^\ast A_i$ belong to the $z$-axis.

On the other hand, since $|A_1|=\lfloor \ep p \rfloor$ and $\lfloor 2/\ep \rfloor \le m_1 \le |A_1|-2/\ep$, Theorem \ref{theorem:ErdosHeilbronn} implies that $m_1^\ast A_1$ contains a translate of $H$, i.~e.~, $m_1^\ast A_1$ contains the whole line which has image $m_1v_1$ in the $xy$-plane.

It follows  that $\sum_{i=1}^n m_i^\ast A_i$ covers the whole $z$-axis, and hence the origin, a contradiction.

{\bf Acknowledgements.} The authors would like to thank the referees for carefully reading this manuscript and providing very helpful remarks.

\end{document}